\documentclass[12pt, reqno]{amsart}

\usepackage[dvipsnames]{xcolor}
\usepackage{tikz}
\usepackage{stmaryrd}
\usetikzlibrary{matrix,arrows,decorations.pathmorphing}
\usepackage{amsfonts}
\usepackage{pdfpages}
\usepackage{graphicx}
\usepackage{mathpazo}
\usepackage{euler}
\usepackage{amssymb}
\usepackage{amsmath,mathtools}
\usepackage{fullpage}
\usepackage{caption}
\usepackage{amsthm}
\usepackage{thmtools}
\usepackage{enumerate}
\usepackage{picinpar} % for pictures in paragraphs
\usepackage{tikz-cd}
\usepackage{color}
\usepackage{latexsym}
\usepackage{enumitem} 
\usepackage[toc,page]{appendix}
\usepackage{url} 
\usepackage{multicol}
\usepackage{rotating}
\usepackage[numbers]{natbib}         % bibliography style
\usepackage[colorlinks]{hyperref} 
\hypersetup{colorlinks=true, linkcolor=red, citecolor = OliveGreen, urlcolor = black}
\makeatletter
\newcommand{\customlabel}[2]{%
   \protected@write \@auxout {}{\string \newlabel {#1}{{#2}{\thepage}{#2}{#1}{}} }%
   \hypertarget{#1}{#2}}
\makeatother
\usepackage{setspace}
\usepackage{mathrsfs}      % for \mathscr

\usepackage{setspace}
\usepackage{indentfirst}           % to indent the first paragraph of a section
\newcommand{\RR}{\mathbb{R}}      % for Real numbers
\newcommand{\QQ}{\mathbb{Q}}

\newcommand{\Qp}{\mathbb{Q}_p}

\newcommand{\Qbar}{\overline{\mathbb{Q}}}

\newcommand{\Fp}{\mathbb{F}_p}

\newcommand{\Pone}{\mathbb{P}^1} % for projective n-space

\newcommand{\mcO}{\mathcal{O}}

\newcommand{\mcF}{\mathcal{F}}

\usepackage{listings}
\usepackage{comment}
\usepackage{manfnt}
\usepackage{longtable}
\makeatletter
\def\@tocline#1#2#3#4#5#6#7{\relax
  \ifnum #1>\c@tocdepth % then omit
  \else
    \par \addpenalty\@secpenalty\addvspace{#2}%
    \begingroup \hyphenpenalty\@M
    \@ifempty{#4}{%
      \@tempdima\csname r@tocindent\number#1\endcsname\relax
    }{%
      \@tempdima#4\relax
    }%
    \parindent\z@ \leftskip#3\relax \advance\leftskip\@tempdima\relax
    \rightskip\@pnumwidth plus4em \parfillskip-\@pnumwidth
    #5\leavevmode\hskip-\@tempdima
      \ifcase #1
       \or\or \hskip 1em \or \hskip 2em \else \hskip 3em \fi%
      #6\nobreak\relax
    \dotfill\hbox to\@pnumwidth{\@tocpagenum{#7}}\par
    \nobreak
    \endgroup
  \fi}
\makeatother

\makeatletter
\def\@tocline#1#2#3#4#5#6#7{\relax
  \ifnum #1>\c@tocdepth % then omit
  \else
    \par \addpenalty\@secpenalty\addvspace{#2}%
    \begingroup \hyphenpenalty\@M
    \@ifempty{#4}{%
      \@tempdima\csname r@tocindent\number#1\endcsname\relax
    }{%
      \@tempdima#4\relax
    }%
    \parindent\z@ \leftskip#3\relax \advance\leftskip\@tempdima\relax
    \rightskip\@pnumwidth plus4em \parfillskip-\@pnumwidth
    #5\leavevmode\hskip-\@tempdima
      \ifcase #1
       \or\or \hskip 1em \or \hskip 2em \else \hskip 3em \fi%
      #6\nobreak\relax
    \dotfill\hbox to\@pnumwidth{\@tocpagenum{#7}}\par
    \nobreak
    \endgroup
  \fi}
\makeatother

\numberwithin{equation}{subsection}

\numberwithin{equation}{subsection}
\newtheorem{theorem}[subsection]{Theorem}
\newtheorem{ass}[subsection]{Assumption}
\newtheorem{lemma}[subsection]{Lemma}
\newtheorem{coro}[subsection]{Corollary}

\newtheorem{prop}[subsection]{Proposition}

\theoremstyle{definition}
\newtheorem{defn}[subsection]{Definition}
\newtheorem{exam}[subsection]{Example}

\newtheorem{remark}[subsection]{Remark}

\theoremstyle{remark}

\newcommand{\mZ}{\mathbb{Z}}
\newcommand{\mR}{\mathbb{R}}
\newcommand{\mQ}{\mathbb{Q}}
\newcommand{\mF}{\mathbb{F}}
\newcommand{\mC}{\mathbb{C}}

\newcommand{\mA}{\mathbb{A}}

\newcommand{\mP}{\mathbb{P}}

\newcommand{\cE}{\mathcal{E}}

\newcommand{\brk}[1]{ \left\lbrace #1 \right\rbrace }

\newcommand{\pwr}[1]{\left( #1 \right)}

\newcommand{\lrra}{\longrightarrow}

\newcommand{\floor}[1]{\left\lfloor #1 \right\rfloor}
\newcommand{\fps}[1]{\llbracket #1 \rrbracket}

\def\quotient#1#2{\raise1ex\hbox{$#1$}{\Large/} \lower1ex\hbox{$#2$}}

\DeclareMathOperator{\tors}{tors}

\DeclareMathOperator{\Jac}{J}

\DeclareMathOperator{\MV}{MV}
\DeclareMathOperator{\New}{New}

\DeclareMathOperator{\Trop}{Trop}

\DeclareMathOperator{\conv}{conv}

\newcommand{\Cp}{\mathbb{C}_p}
\newcommand{\Fpbar}{\overline{\Fp}}
\newcommand{\Pbar}{\overline{P}}

\begin{document}
%\begin{titlepage}
\title{Irrational points on hyperelliptic curves
%Bounding unexpected quadratic and cubic points in families of odd degree hyperelliptic curves over $\mQ$
}

\author{Joseph Gunther}
\address{Department of Mathematics, Universit\'e Paris-Sud, Orsay, France F-91405}
\email{jgunther7@gmail.com}

\author{Jackson S. Morrow}
\address{Department of Mathematics, Emory University,
Atlanta, GA 30322}
\email{jmorrow4692@gmail.com}

\begin{abstract}
We consider genus $g$ hyperelliptic curves over $\mathbb{Q}$ with a rational Weierstrass point, ordered by height. 
If $d<g$ is odd, we prove, under an assumption, that there exists $B_d$ such that a positive proportion of these curves have at most $B_d$ points of degree $d$. Similarly, if $d<g$ is even, we conditionally bound degree $d$ points not pulled back from points of degree $d/2$ on the projective line. 
Furthermore, we show one may take $B_2 = 24$ and $B_3 = 114$. 

Our proofs proceed by refining recent work of Park, which applied tropical geometry methods to symmetric power Chabauty, and then applying results of Bhargava and Gross on average ranks of Jacobians of hyperelliptic curves. 
%We prove that for each genus $g \geq 3$, a positive proportion of these curves have at most 24 quadratic points not obtained by pulling back rational points of the projective line.  We also show that for each $g\geq 5$, a positive proportion have at most 144 cubic points.  
\end{abstract}
\date{\today}
\maketitle

\section{Introduction}\label{Intro}
Let $K$ be a number field, and let $C/K$ be a curve of genus $g \geq 2$.  In 1983, Faltings \cite{faltingsmordell} proved that the set $C(K)$ of $K$-rational points is finite.  Given this, one can ask how the finite quantity $\#C(K)$ varies in families of curves.  Recently, multiple works have considered this question, for the family of all hyperelliptic curves over $\QQ$ with a rational Weierstrass point \cite{poonen2014most, romanothorne}, the family with a rational non-Weierstrass point \cite{shankarwang}, and the entire family of hyperelliptic curves over $\QQ$ \cite{bhargava2013most}.

In this paper, for a hyperelliptic curve $C/\QQ$, instead of rational points we consider the \textit{degree} $d$ points of $C$, which we take to mean the set $$\brk{P \in C(\Qbar) \ \left| \ [\QQ(P):\QQ] = d\right.}.$$  Since $C$ is defined over $\QQ$, this set is partitioned into $d$-tuples of Galois-conjugate points.

Before stating our main theorems, we consider an extended example: quadratic points, i.e.~$d=2$. Because we allow the quadratic extension to vary, there are infinitely many such points: for almost any point of $\Pone(\QQ)$, its pre-image under the hyperelliptic map will be a pair of conjugate quadratic points on $C$.  We will call these \textit{expected} quadratic points. More simply, for a hyperelliptic equation $y^2=f(x)$, these are the quadratic points given by plugging in a rational number for $x$, and then solving for $y$.

But there can also be \textit{unexpected} quadratic points, whose $x$-coordinate is irrational but whose $y$-coordinate is contained in the same quadratic field.  For example, the genus 4 curve defined by $y^2=x^9+x^3-1$ contains infinitely many expected points, such as $(0,\pm i), (-1, \pm \sqrt{-3})$, and $(2,\pm \sqrt{519})$, but also contains unexpected points like $(\pm i, \pm i)$, $(\zeta_3, \pm 1)$, and $(-\zeta_3, \pm \sqrt{-3})$, where $\zeta_3$ is a primitive third root of unity.  In general, it is no small feat to compute these points explicitly for a given curve.

\begin{exam}
Let $C/\mQ$ be the hyperelliptic curve with affine model given by
$$C\colon y^2 = f(x) = x(x^2 + 2)(x^2 + 43)(x^2 + 8x - 6).$$
In \cite[Section~6.1]{siksek2009chabauty}, Siksek determined the set of quadratic points on $C$. Besides the infinitely many expected quadratic points of the form $(x,\pm \sqrt{f(x)})$, for $x \in \QQ$, there are exactly 9 pairs of unexpected quadratic points on $C$. For example, there are the three pairs below:
\begin{align*}
\mathcal{Q}_1 &= \brk{(\sqrt{6},56\sqrt{6}),(-\sqrt{6},-56\sqrt{6})},  \mathcal{Q}_2 = \brk{(\sqrt{-2},0),(-\sqrt{-2},0)}, & \\
\mathcal{Q}_3 &= \brk{\pwr{\frac{-164 + \sqrt{22094}}{49},\frac{257704352 - 1648200\sqrt{22094}}{823543}},\text{ conjugate}}. & 
\end{align*}
\end{exam}

\begin{exam}
The modular curve $X_0(29)$ is a genus 2 curve with affine model 
\[
y^2 + (-x^3 - 1)y = -x^5 - 3x^4 + 2x^2 + 2x - 2.
\]
In \cite[Table 5]{bruin_najman_2015}, Bruin and Najman determined that there are exactly 4 pairs of unexpected quadratic points on $X_0(29)$:
\begin{align*}
P_1 &= \pwr{\sqrt{-1} - 1, 2\sqrt{-1} + 4}, & P_2 &= \pwr{\sqrt{-1} - 1, \sqrt{-1} - 1}, \\
P_3 &= \pwr{\frac{1}{4}(\sqrt{-7} + 1), \frac{1}{16}(-11\sqrt{-7} - 7)}, & P_4 &= \pwr{\frac{1}{4}(\sqrt{-7} + 1), \frac{1}{8}(5\sqrt{-7} + 9)},
\end{align*}
along with their respective images under the hyperelliptic involution.

Addtional works of Ozman--Siksek \cite{ozman_siksek2018} and Box \cite{box2019} have classified the unexpected quadratic points on the modular curves $X_0(N)$ for various values of $N$.
%\[
%N \in 
%\brk{
%\begin{tabular}{c}
%22, 23, 26, 28, 29, 30, 31, 33, 34, 35, 37, 38, 39, 40, 41, 42,43, 44, 45, 46, 47, \\
%48, 50,51, 52, 53, 54, 55, 56,  57, 59, 61, 64 , 64 , 65, 67, 71, 72, 73, 75, 81
%\end{tabular}
%}.
%\]
\end{exam}

The examples above demonstrate a general phenomenon: by further work of Faltings \cite[p.~550]{faltings91}, we know that for any hyperelliptic curve of genus $g \geq 4$, there are only finitely many of these \textit{unexpected} quadratic points. Thus, one can ask how many there are on a typical hyperelliptic curve.

To make that question rigorous, we need a way of ordering curves. A genus $g$ hyperelliptic curve $C$ over $\QQ$ has a marked rational Weierstrass point $\infty$ if and only if it can be given an affine model of the form 
\begin{equation}\label{eqn:oddhyperelliptic}
y^2=f(x)=x^{2g+1} + a_2x^{2g-1} + a_3x^{2g-2} +  \dots + a_{2g+1},
\end{equation}
with $f(x) \in \mathbb{Z}[x]$ separable, such that $\infty$ is not contained in this affine patch. Furthermore, $C$ has a unique such \textit{minimal} equation, for which there is no prime $p$ such that $p^{2i} \mid a_i$ for each $i \geq 2$.  Define the \textit{height} of $C$ to be $$H(C)\coloneqq\max\brk{|a_i|^{1/i}},$$ where the $a_i$'s are coefficients for the minimal equation of $C$.  

To study this question, we use and refine work of Park \cite{park16}, which uses tropical intersection theory. However, the techniques of that pre-print seem to be missing a technical hypothesis, along the lines of ($\dagger$) in our Assumption \ref{ass:Park}, in order to be valid; see Section \ref{sec:proofmain} and Remark \ref{remark:assdiscussion} for an explanation.  This is the source of the conditional nature of some of our results.

We can now state our first conditional theorem.

\begin{theorem}\label{hyperthm}
Suppose Assumption \ref{ass:Park} holds. Then for each $g > 2$, a positive proportion of genus $g$ hyperelliptic curves over $\QQ$ with a rational Weierstrass point, when ordered by height, have at most 24 quadratic points not obtained by pulling back points of $\Pone(\QQ).$
\end{theorem}

%\begin{remark}
%Assumption \ref{ass:Park} is an artifact of using tropical geometry in symmetric power Chabauty--Coleman to provide explicit bounds for higher degree points on curves. 
%We refer the reader to Remark \ref{remark:assdiscussion} for further discussion.
%\end{remark}
%changed to numbered theorem for submission to IMRN
%\begin{thmx}\label{hyperthm}
%For each $g \geq 4$, a positive proportion of genus $g$ hyperelliptic curves over $\QQ$ with a rational Weierstrass point, when ordered by height, have at most 12 pairs of quadratic points not obtained by pulling back points of $\Pone(\QQ).$
%\end{thmx}

More precisely, let $\mcF_g$ denote the set of $\QQ$-isomorphism classes of genus $g$ hyperelliptic curves defined over $\QQ$, with a marked rational Weierstrass point.  The above says that if $\mcF'_g \subset \mcF_g$ corresponds to those curves satisfying the conditions of Theorem \ref{hyperthm}, then $$\liminf_{X \to \infty} \frac{\#\{C \in \mcF'_g \ | \ H(C) < X\}}{\#\{C \in \mcF_g \ | \ H(C) < X\}} >0.$$

The bound of Theorem \ref{hyperthm} does not hold for all hyperelliptic curves, as shown by the following example, told to us by Michael Stoll.
\begin{exam}
Let $f_1(x), \ldots, f_9(x)$ be distinct irreducible monic quadratic polynomials with rational coefficients.  Write their product as the square of a degree 9 polynomial, plus a remainder polynomial $r(x)$ of degree at most 8. Then $y^2=-r(x)$ will usually have at least 36 ``unexpected" quadratic solutions.
\end{exam}

Next we consider cubic points, i.e.~degree 3 points. Unlike the case of $d=2$, where the geometry --- in this case, the existence of a 2:1 map to $\Pone$ --- imposes infinitely many quadratic points, there need not be any cubic points.  We prove the following theorem on their sparsity.

\begin{theorem}\label{cubicthm}
Suppose Assumption \ref{ass:Park} holds. Then for each $g > 3$, a positive proportion of genus $g$ hyperelliptic curves over $\QQ$ with a rational Weierstrass point, when ordered by height, have at most 114 cubic points.
\end{theorem}

The $d=2$ and $d=3$ cases turn out to be prototypical, and we can now state our main theorem concerning points of arbitrary degree $d$.

\begin{theorem}\label{thm:degreedbound}
Let $d>1$ be a positive integer and suppose Assumption \ref{ass:Park} holds. 
\begin{enumerate}
\item If $d$ is odd, there exists a number $B_d$ such that for each $g > d$, a positive proportion of genus $g$ hyperelliptic curves over $\QQ$ with a rational Weierstrass point have at most $B_d$ points of degree $d$.

\item If $d$ is even, there exists a number $B_d$ such that for each $g > d$, a positive proportion of genus $g$ hyperelliptic curves over $\QQ$ with a rational Weierstrass point have at most $B_d$ points of degree $d$ not obtained by pulling back degree $\frac{d}{2}$ points of $\Pone$.

\item We may take $B_2 = 24$ and $B_3 = 114$.
\end{enumerate}
\end{theorem}

\begin{remark}
\begin{itemize}
\item []
\item If $d=1$, we may unconditionally take $B_1=1$, by \cite[Theorem 10.3]{poonen2014most} (in the case $g>2$) and \cite[Theorem 1.2]{romanothorne} ($g=2$).
\item The hypothesis $g > d$ is natural: that is exactly when the symmetric product $C^{(d)}$ is of general type, for a curve of genus $g \geq 2$. Since a degree $d$ point of $C$ gives a rational point of $C^{(d)}$, this is when one would expect them to be rare.
%\item Using our methods, we could \textit{unconditionally} prove $B_2=0$ if we knew that a positive proportion of odd hyperelliptic curves had Jacobian's of rank 0. In other words, we would know that a positive proportion have no unexpected quadratic points.
\end{itemize}
\end{remark}

\subsection{Terminology}\label{subsec:term}
While we do not use the term in theorem statements, we will call a point of even degree $d$ on a hyperelliptic curve \textit{unexpected} if it does not map to a degree $d/2$ point on $\mP^1$. If $d$ is odd, we call any point of degree $d$ \textit{unexpected}.

Throughout the paper, we use ``asymptotically" when considering hyperelliptic curves of fixed genus with a marked rational Weierstrass point, with increasing height and ``congruence conditions" when considering the coefficients of the minimal equation in \eqref{eqn:oddhyperelliptic}.

\subsection{Overview of paper}\label{subsec:overviewquadratic}
In Section \ref{sec:arithmeticJacobians}, we collect results on the arithmetic of hyperelliptic Jacobians and prove that for a fixed genus $g\geq 2$, asymptotically $100\%$ of hyperelliptic curves over $\mQ$ of genus $g$ with a rational Weierstrass point have 
%geometrically simple Jacobian (cf.~Proposition \ref{simple}). We then prove that for such hyperelliptic curves, there are 
finitely many unexpected degree $d$ points.
% not corresponding to exceptional special divisors, i.e.~there are finitely many \textit{unexpected} degree $d>1$ points (cf.~Proposition \ref{prop:finitenessdegreedpoints}). 
%Another aspect of Park \cite{park16} gives an effective bound $B(p,g)$ on the number of these unexpected degree $d$ point for such hyperelliptic curves which also have good reduction at $p$ and whose Jacobian has rank no more than 1. 

The main part of our proof of Theorem \ref{thm:degreedbound} involves refining recent work of Park \cite{park16} that partially generalized the Chabauty--Coleman method for computing rational points on curves of genus $g\geq 2$ to higher degree points.  
One aspect of her work yields an effective bound $B(p,g)$ on the number of unexpected degree $d$ points on a hyperelliptic curve $C$ as above with the additional conditions: the Mordell--Weil rank of the curve's Jacobian is no more than 1, $C$ has good reduction at $p$, and there exists a basis of 1-forms $\brk{\omega_1,\dots,\omega_d}$ which satisfy a genericity condition on the tropicalization of their local expansions as $p$-adic integrals. 
%(cf.~Section \ref{sec:padicprem}).
Besides clarifying the need for that technical assumption, our paper removes the dependence of the above result on the genus $g$, in order to get constants that only depend on $d$ in Theorems \ref{hyperthm}, \ref{cubicthm}, and \ref{thm:degreedbound}.

In Section \ref{sec:padicprem}, we recall results on $p$-adic integration and prove some auxiliary lemmas. In Section \ref{sec:analyticloci}, we begin removing the dependence on the genus from Park's bound above.
%(cf.~Lemmas \ref{universalbasis} and \ref{bestuniversalbasis})
%Then we choose, for each $g$, a positive-density family of odd hyperelliptic curves with good reduction at a specified prime $p$.
Work of Bhargava and Gross \cite{bhargavagross} tells us that many curves have Jacobian of rank at most $1$, and thus our new bound applies to them. 
In Section \ref{sec:proofmain}, using results from tropical geometry, we prove Theorem \ref{thm:degreedbound} using Newton polygon and mixed volume computations. 
In Sections \ref{boundingquadratic} and \ref{boundingcubic}, we prove our explicit results for $d=2$ and $3$.

Many of our results hold unconditionally for any individual curve that satisfies the tropical technical hypothesis mentioned above.  The conditional part of the paper is the assumption that enough curves do satisfy it.
%(cf.~Lemma \ref{lem:Fpreductionbound} and Sections \ref{boundingquadratic} and \ref{boundingcubic})

\subsection{Related results}
We conclude the introduction by describing some related results in the literature. All curves below are defined over $\QQ$.

In \cite{bhargavagross}, Bhargava and Gross showed that for $g \geq 2$, a positive proportion of genus $g$ hyperelliptic curves with a marked rational Weierstrass point have at most 3 rational points, and that a majority have less than 20 rational points.  In \cite{poonen2014most}, Poonen and Stoll showed that in fact, for $g \geq 3$, a positive proportion of these curves have no other rational points besides the marked point.  (Romano and Thorne \cite{romanothorne} recently proved this for $g=2$.) Furthermore, that proportion of curves tends to 1 as $g$ grows.

Shortly after Poonen and Stoll's work, Shankar and Wang \cite{shankarwang} proved that for $g \geq 9$, a positive proportion of genus $g$ hyperelliptic curves with a marked rational non-Weierstrass point have only the two guaranteed rational points (the marked point and the other point in its fiber).  Again, that proportion tends to 1 as $g$ grows.  Next, Bhargava \cite{bhargava2013most} showed that for $g \geq 2$, a positive proportion (again tending to 1) of all genus $g$ hyperelliptic curves have no rational points.

The question of higher-degree points has been considered previously for families of hyperelliptic curves, though only in the case of odd-degree points.  While every hyperelliptic curve has (infinitely many!)~points of each even degree, the geometry of a general hyperelliptic curve does not impose any points of odd degree.  In \cite{bhargavagrosswang}, Bhargava, Gross, and Wang showed that for $g \geq 2$, a positive proportion of all locally soluble genus $g$ hyperelliptic curves have no points over any odd-degree extension of $\QQ$.  Moreover, for a fixed odd $m$, the proportion of locally soluble curves without a degree $m$ point tends to 1 as $g$ grows.

Bhargava, Gross, and Wang's results work by showing that many curves have no rational odd-degree divisors at all.  The curves in the family considered in this paper (those with a rational Weierstrass point) always have such divisors, so our Theorem \ref{cubicthm} is disjoint from (but complementary to) their work.  In particular, since degree $d$ points on a curve correspond to rational points of the $d$-th symmetric power of $C$, our results represent some of the first work on bounding rational points in families of higher-dimensional varieties that do have {\em some} rational points. %In the sense that Unlike the case of quadratic points, the geometry of a general hyperelliptic curve does not impose any points of odd degree, and so their result can be proved via local considerations.

%Since every hyperelliptic curve over $\QQ$ has (infinitely many!) quadratic points, we cannot use purely local considerations; our global input is the work of Bhargava and Gross on Jacobian ranks.

%Lastly, Ananth Shankar \cite{ananth} has results on average Selmer group ranks that should allow one to make similar statements for hyperelliptic curves with a marked rational Weierstrass point \textit{and} a marked rational non-Weierstrass point. In a different direction, one can ask analogous questions for elliptic curves.  Of course, as soon as the rank is positive, there will be infinitely many rational points, so one cannot ask the exact same questions as for hyperelliptic curves. A theorem of Siegel asserts that an elliptic curve always has finitely many integral points, and hence one can ask how many of these point there are on average.  Alpoge \cite{alpoge} has recently shown that the average number of integral points on an elliptic curve is at most 66.

%moved to end for submission to IMRN
%\subsection{Acknowledgments} 
%We thank Johan de Jong, Bastian Haase, Joe Kramer-Miller, Jennifer Park, Bjorn Poonen, and David Zureick-Brown for helpful discussions. 
%We also thank Eric Katz and Jerry Wang for useful comments on an earlier draft. 
%This research began at the ``Chip Firing and Tropical Curves" summer graduate school, organized by Matt Baker, David Jensen, and Sam Payne at the Mathematical Sciences Research Institute. 

\section{Arithmetic and geometry of hyperelliptic Jacobians}\label{sec:arithmeticJacobians}
First, we recall results of Bhargava and Gross on the average size of 2-Selmer groups of Jacobians of hyperelliptic curves.

\begin{theorem}[\protect{\cite[Theorem~1.1]{bhargavagross}}]
When all hyperelliptic curves of fixed genus $g \geq 1$ over $\QQ$ having a rational Weierstrass
point are ordered by height, the average size of the 2-Selmer groups of their Jacobians is at most 3.

Furthermore, the same result holds if one averages within a family defined by a finite set of congruence conditions.
\end{theorem}

This gives immediate corollaries concerning the average rank of $\Jac(\QQ)$, where we write $\Jac$ for the Jacobian of a curve $C$.

\begin{coro}[\protect{\cite[Corollary~1.2]{bhargavagross}}]
When all hyperelliptic curves of fixed genus $g \geq 1$ over $\QQ$ having a rational Weierstrass
point are ordered by height, the average rank of the Mordell--Weil groups of their Jacobians is at most $\frac{3}{2}$.

Furthermore, the same result holds if one averages within a family defined by a finite set of congruence conditions.
\end{coro}

\begin{coro}\label{goodrankcor}
When all hyperelliptic curves of fixed genus $g \geq 1$ over $\QQ$ having a rational Weierstrass
point are ordered by height, at least $25\%$ have $\operatorname{rank} \Jac(\QQ) =0$ or $1$.  The same holds if one averages within a congruence family.

Furthermore, either a positive proportion have rank 0, or at least $50 \%$ have rank 1.
\end{coro}

%Next, we recall some results on the geometry and arithmetic of hyperelliptic curves and their Jacobians.  A curve of genus $g\geq 2$ is hyperelliptic in at most one way \cite[Lemma 3.3]{accola}, and a curve of genus $g \geq 4$ cannot be simultaneously hyperelliptic and bielliptic (a 2-to-1 cover of an elliptic curve); this follows immediately from the Castelnuovo--Severi inequality \cite[Theorem 3.5]{accola}.  Curves of genus 2 or 3 can be simultaneously hyperelliptic and bielliptic.
%one can easily cook up genus 2 examples as $y^2 = (\text{sextic with only even powers of }x$), and genus 3 examples as curves of bidegree $(2,4)$ on $\Pone \times \Pone$.

Next, we recall a deep theorem of Faltings about rational points on subvarieties of abelian varieties.

\begin{theorem}[\protect{\cite[p.~175]{faltings94}}]\label{faltingsabelian}
Let $X$ be a closed subvariety of an abelian variety $A$, with both defined over a number field $K$.  Then the set $X(K)$ equals a finite union $\cup B_i(K)$, where each $B_i$ is a translated abelian subvariety of $A$ contained in $X$.
\end{theorem}

To conclude this section, we prove that asymptotically, $100\%$ of hyperelliptic curves with a rational Weierstrass point over $\mQ$ have finitely many \emph{unexpected} degree $d$ points, as described in Theorem \ref{thm:degreedbound}. 

\begin{prop}\label{simple}
Fix $g \geq 2$. Asymptotically, 100\% of genus $g$ hyperelliptic curves over $\QQ$ with a rational Weierstrass point have geometrically simple Jacobian.
\end{prop}
\begin{proof}
Let $t_2,\dots,t_{2g+1}$ be indeterminates.  The polynomial $$F(x,t_2,\dots, t_{2g+1})=x^{2g+1} + t_2x^{2g-1} + \dots + t_{2g+1}$$ has Galois group $S_{2g+1}$ over the field $\QQ(t_2,\dots,t_{2g+1})$.  One way to see this is to note that its specialization at $t_2=\dots=t_{2g-1}=0$ and $t_{2g}=t_{2g+1}=-1$ gives $x^{2g+1}-x-1$, which has Galois group $S_{2g+1}$, by \cite[Corollary 3]{osada} or \cite{nartvila}.

By a criterion of Zarhin \cite[Theorem 1]{zarhin}, the Jacobian of the hyperelliptic curve given by $y^2=f(x)$ is geometrically simple whenever $f(x)$ has Galois group $S_{\deg f}$. 

Let $\cE=\cE_g$ be the complement in $\mA^{2g}$ of the discriminant locus for equations of the form $y^2=x^{2g+1}+a_{2}x^{2g-1}+ \dots + a_{2g+1}$.  We apply a version of Hilbert irreducibility (our height weights coordinates unequally, so some care must be taken); see \cite[Theorem 2.1]{cohen}, adapted as in \cite[Section 5, Notes (iii)]{cohen}.  It implies that asymptotically 100\% of all the integer points in $\cE$, when ordered by height, give hyperelliptic curves whose Jacobians are geometrically simple.  A sieving argument shows that a positive proportion of the integer points of $\cE$ give minimal equations, so asymptotically 100\% of minimal equations give curves with geometrically simple Jacobians.
\end{proof}

For a curve $C$ and a positive integer $d$, let $C^{(d)}$ denote its $d$-th symmetric product, the points of which correspond to effective degree $d$ divisors on $C$.  Note that a conjugate $d$-tuple of points on $C$ gives a rational point of the symmetric product.

\begin{prop}\label{prop:finitenessdegreedpoints}
Let $d$ be a positive integer, and let $g>d$ be an integer.
\begin{enumerate}
\item If $d$ is odd, then asymptotically, 100\% of genus $g$ hyperelliptic curves over $\QQ$ with a rational Weierstrass point have finitely many degree $d$ points.
\item If $d$ is even, then asymptotically, 100\% of genus $g$ hyperelliptic curves over $\QQ$ with a rational Weierstrass point have finitely many degree $d$ points not obtained by pulling back degree $\frac{d}{2}$ points of $\Pone$.
\end{enumerate}
\end{prop}

\begin{proof}
First, note that since the map from a hyperelliptic curve $C$ to $\Pone$ has degree two, the image of a $d$-tuple of conjugate points on $C$ is either a $d$-tuple of conjugate points on $\Pone$, or possibly a $\frac{d}{2}$-tuple of conjugate points on $\Pone$ if $d$ is even.

We may assume $C$ has geometrically simple Jacobian $J$, by Proposition \ref{simple}.  Map the symmetric product $C^{(d)}$ to $J$ via the Abel--Jacobi map given by the rational Weierstrass point, i.e.~$P_1 + \cdots + P_d \mapsto [P_1 + \cdots + P_d - d\cdot \infty].$

Since $d<g$, the image $W_d$ is a proper closed subvariety of $J$.  Since $J$ is geometrically simple, $W_d$ contains no translate of a positive-dimensional abelian subvariety of $J$.  By Theorem \ref{faltingsabelian}, $W_d(\QQ)$ is finite.

Lastly, we can ignore the positive-dimensional fibers of $C^{(d)} \rightarrow J$, which correspond to effective degree $d$ divisors $D$ on $C$ such that $D$ has positive rank.  On a hyperelliptic curve, any such divisor $D$ must contain a subdivisor of the form $P + \iota(P)$, where $P$ is some point of $C$ and $\iota$ is the hyperelliptic involution, which switches points within the same fiber \cite[p.~13]{acgh}.  

But by the first paragraph, if $d$ is odd, no $d$-tuple of conjugate points can make up such a $D$. If $d$ is even, it is only possible if $$D=P_1+\iota(P_1) + \dots + P_{\frac{d}{2}} + \iota(P_{\frac{d}{2}}),$$ which will map to a $\frac{d}{2}$-tuple of conjugate points on $\mP^1$. Thus in either case, the set we wish to show is finite injects into the finite set $W_d(\QQ)$.
\end{proof}

\section{$p$-adic preliminaries}\label{sec:padicprem}
We recall some results on $p$-adic integration and the Chabauty--Coleman method; we refer the reader to \cite{mccallum2007method, siksek2009chabauty, park16} for a fuller account of these techniques. We also prove some auxiliary lemmas.

\subsection{Vanishing of integrals}\label{subsec:vanishingintegrals}
Fix $C/\QQ$ a curve of genus $g \geq 2$, and $p$ a prime number.  We make use of $p$-adic integration on the Jacobian variety $J$ of our curve.  Let $\Cp$ be the completion of the algebraic closure of $\Qp$.  After we base change from $\QQ$ to $\Qp$, we have an integration pairing
\begin{align*}
H^0(C_{\Qp},\Omega^1) \times J(\Cp) \rightarrow \Cp\\
(\omega, D) \mapsto \int_0^D\omega
\end{align*}
that is $\Qp$-linear in the left factor, and a group homomorphism in the right.  The kernel on the left is trivial, and on the right is the torsion subgroup $J(\Cp)_{\tors}$.

Let $r$ be the rank of $J(\QQ)$ as a finitely generated abelian group (for the rest of the paper, $r$ will denote this rank for whatever curve is under consideration).  We identify $J(\QQ)$ with its image in $J(\Qp)$ and $J(\Cp)$.  Within the former, its $p$-adic closure $\overline{J(\QQ)} \subset J(\Qp)$ will be a finitely generated $\mZ_p$-module of rank at most $r$.  Define $$\Lambda_C \coloneqq\brk{\omega \in H^0(C_{\Qp},\Omega^1) \ \left| \ \int_{0}^D\omega = 0 \text{ for all } D \in J(\QQ)\right.}.$$ This is a $\Qp$-vector space of dimension at least $g-r$.

Suppose further that $p$ is a prime of good reduction for our curve $C$.  For a point $P \in C(\Cp)$, let $\Pbar \in C_{\Fp}(\Fpbar)$ denote its reduction at $p$. Then given a nonzero form $\omega \in H^0(C_{\Qp},\Omega^1)$, we can scale it by an element of $\Qp^\times$ to give a {\em normalized} form, which we take to mean it reduces to a nonzero element $\overline{\omega} \in H^0(C_{\Fp}, \Omega^1)$.  For a normalized form $\omega$, and a point $\Pbar \in C_{\Fp}(\Fpbar)$, we define $n(\omega, \Pbar)$ to be the order of vanishing of $\overline{\omega}$ at $\overline{P}$.  As long as $\Lambda_C \neq \{0\}$, we then define $$n(\Lambda_C, \Pbar) = \min_{\text{normalized }\omega \in \Lambda_C} n(\omega, \Pbar).$$ By \cite[Theorem 6.4]{stoll2006independence}, the lower the rank is, the better we can control these minimal orders of vanishing.
\begin{theorem}[Stoll]\label{stollbound}
Let $C/\QQ$ be a curve of genus $g \geq 2$, with rank $r \leq g-1$, and let $p$ be a prime of good reduction. Then $\sum_{\Pbar \in C_{\Fp}(\Fpbar)} n(\Lambda_C,\Pbar) \leq 2r.$
\end{theorem}

For our purposes, we need forms that achieve these minima at different points simultaneously.

\begin{lemma}\label{universalform}
Let $C/\QQ$ be a curve of genus $g \geq 2$, and let $p$ be a prime of good reduction. Let $\overline{P_1}, \dots, \overline{P_d} \in C_{\Fp}(\Fpbar)$. Suppose $r \leq g-1$ and $p \geq d$.  Then there exists a normalized $\omega \in \Lambda_C$ such that $n(\Lambda_C, \overline{P_i}) = n(\omega, \overline{P_i})$ for $i=1, \dots, d$.
\end{lemma}
\begin{proof}
We proceed by induction on $d$. The base case $d=1$ is immediate from the definition of $n(\Lambda_C, \overline{P_1}).$ Suppose there is a normalized form $\omega'$ such that $n(\Lambda_C, \overline{P_i}) = n(\omega', \overline{P_i})$ for $i=1, \dots, d-1.$  If $n(\Lambda_C, \overline{P_d}) = n(\omega', \overline{P_d})$, we may take $\omega=\omega'$.  Otherwise, choose a normalized $\omega''$ such that $n(\Lambda_C, \overline{P_d}) = n(\omega'', \overline{P_d}).$ If $n(\Lambda_C, \overline{P_i}) = n(\omega'', \overline{P_i})$ for $i=1, \dots, d-1,$ we may take $\omega=\omega''$.

So suppose without loss of generality that $n(\omega'', \overline{P_1}) > n(\Lambda_C, \overline{P_1}).$ Let $t_2, \dots, t_{d-1}$ be uniformizers at $\overline{P_2}, \dots, \overline{P_{d-1}}$, respectively.  Write both $\overline{\omega'}$ and $\overline{\omega''}$ with respect to each uniformizer:
\begin{align*}
\overline{\omega'} = a_it_i^{n_i}dt_i, &\\
\overline{\omega''} = b_it_i^{n_i}dt_i, & \textup{ for } i=2,\dots,d-1,
\end{align*}
where for each $a_i, b_i \in \Fpbar(C_{\Fp})$, the geometric function field of the reduction, we have $0=v_{\overline{P_i}}(a_i) \leq v_{\overline{P_i}}(b_i).$ Since $p \geq d$, there exists $0 \neq \alpha \in \Fp$ such that $\alpha \cdot b_i(\overline{P_i}) \neq -a_i(\overline{P_i})$ for $i=2, \dots, d-1$.  Choosing any $\tilde{\alpha}\in \mZ_p$ whose reduction mod $p$ is $\alpha$, we may take $\omega=\omega'+\tilde{\alpha} \omega''$.
\end{proof}

\begin{lemma}\label{universalbasis}
Let $C/\QQ$ be a curve of genus $g \geq 2$, and let $p$ be a prime of good reduction. Let $\overline{P_1}, \dots, \overline{P_d} \in C_{\Fp}(\Fpbar)$, and suppose $r \leq g-d$.  Then there exist linearly independent, normalized $\omega_1, \dots, \omega_d \in \Lambda_C$ such that $n(\Lambda_C, \overline{P_i}) = n(\omega_j, \overline{P_i})$ for all $i,j=1, \dots, d$.
\end{lemma}

\begin{proof}
Take $\omega_1$ to be $\omega$ as given by Lemma \ref{universalform}.  By the rank condition, we can choose $\omega_2', \dots, \omega_d' \in \Lambda_C$ to be normalized forms such that $\omega_1, \omega_2', \dots, \omega_d'$ are linearly independent.  Then each reduction $\overline{\omega_1 + p\omega_j'} = \overline{\omega_1}$, so we can take $\omega_j = \omega_1 + p\omega_j'$ for $j=2,\dots,d$.
\end{proof}

\subsection{Vanishing of locally analytic functions}
Now let $C/\QQ$ be a genus $g$ curve with a marked rational point, which we denote by $\infty.$  For any $\omega \in H^0(C_{\Qp},\Omega^1)$, we define (locally analytic) functions 
\begin{align*}
 f_\omega \colon C(\Cp) & \lrra \Cp \\ 
 P& \longmapsto \int_0^{[P-\infty]}\omega,
\end{align*}
and more generally for $d$ a positive integer,  
\begin{align*}
 F_\omega^d \colon C(\Cp) \times \dots \times C(\Cp) &\lrra \Cp\\
 (P_1,\dots,P_d)&\longmapsto f_\omega(P_1)+\dots + f_\omega(P_d) = \int_0^{[P_1+\dots+P_d-d\infty]}\omega.
\end{align*}
The starting point of the Chabauty--Coleman method for examining rational points is that if $\omega \in \Lambda_C$, then for any $P \in C(\QQ)$, we have $f_\omega(P)=0$, because $[P-\infty] \in J(\QQ)$.  The starting point for our method, following \cite{siksek2009chabauty, park16}, is that for $\omega \in \Lambda_C$ and $(P_1,\dots,P_d)$ a $d$-tuple of conjugate degree $d$ points on $C$, we have $F_{\omega}^d(P_1,\dots,P_d)=0$, since $[P_1+\dots+P_d-d\infty] \in J(\QQ)$.

We wish to control these zeros.  For $\Pbar \in C_{\Fp}(\Fpbar)$, define the \textit{residue disk} $$D_{\Pbar} = \brk{Q \in C(\Cp) \ \left| \ \overline{Q} = \Pbar \right.}.$$   Let $D \subset \Cp$ be the open unit disk, i.e.~elements with absolute value strictly less than 1. For $P \in C(\overline{\mQ_p})$, \cite[Lemma 2.3]{siksek2009chabauty} asserts that we can always choose a \textit{well-behaved uniformizer} $z_P$ at $P$, which has the following key properties. First, the function $z_P\colon D_{\Pbar} \rightarrow D$ is a diffeomorphism, with $z_P(P)=0$.  Furthermore, for a finite extension $L/\mQ_p(P)$, with uniformizing element $\pi$, we have that $z_P$ defines a bijection between $C(L) \cap D_{\Pbar}$ and the $\pi$-adic disc $\pi\mcO_L$, given by $Q \mapsto z_P(Q)$.

\begin{remark}\label{valuationbound}
Let $v$ be the valuation on $\Cp$, normalized so that $v(p)=1$.  For $P, Q \in C(\overline{\mQ_p})$ such that $\Pbar = \overline{Q}$, let $e$ be the ramification degree of $\Qp(P,Q)$.  The above implies that $v(z_P(Q)) \geq \frac{1}{e}$.
%n any quadratic point $Q$ in the residue disk $D_{\Pbar}$ will satisfy $v(z_P(Q)) \geq \frac{1}{2}$.  
%Furthermore, if $Q$ reduces to a point of $C_{\Fp}(\Fptwo) \setminus C_{\Fp}(\Fp)$, then we have $v(z_P(Q)) \geq 1$, since the ramification index and the inertia degree of $\Qp(Q)/\Qp$ must multiply to 2.
\end{remark}

We can formally expand a normalized form $\omega$ with respect to the uniformizer $z_P$, as $$\pwr{\sum_{i=0}^\infty a_iz_P^i}dz_P,$$ where the coefficients live in $\Qp(P)$, and are integral ($v(a_i) \geq 0$ for all $i$).  We record a few important facts from \cite[Section 2]{siksek2009chabauty} about this expansion.  
First, the power series $\sum_{i=0}^\infty a_it^i$ is convergent on $D$.  Second, there is a connection to orders of vanishing: the smallest index $i$ for which we have $v(a_i)=0$ is given by $i=n(\omega, \Pbar)$.  Lastly, for $Q \in D_{\Pbar}$, the restriction of $f_\omega$ to $D_{\Pbar}$ is given by $$f_{\omega}(Q) =\int_0^{[P-\infty]}\omega + \sum_{i=0}^\infty \frac{a_i}{i+1}z_P(Q)^{i+1}.$$

Similarly, for $P_1, P_2 \in C(\overline{\mQ_p})$, the restriction of $F_\omega^2$ to $D_{\overline{P_1}} \times D_{\overline{P_2}}$ is given by $$F_\omega^2(Q_1,Q_2)  = \int_0^{[P_1+P_2-2\infty]}\omega + \sum_{i=0}^\infty \frac{a_i}{i+1}z_{P_1}(Q_1)^{i+1}  + \sum_{i=0}^\infty \frac{b_i}{i+1}z_{P_2}(Q_2)^{i+1} .$$
Analogous expansions of course hold for $F_\omega^d$, for arbitrary $d$.

\section{Analytic loci for low-rank hyperelliptic curves}\label{sec:analyticloci}
Besides the correction (see p.~2 and Section \ref{sec:proofmain}), Park’s \cite{park16} most general results require a second technical hypothesis (\textit{loc.~cit.~}Assumption~1.3) involving excess analytic intersection of the zero loci of the $F_{\omega}^d$'s for $\omega \in \Lambda_C$.  
In this section, we prove an assertion of Park (\textit{loc.~cit.~}p.~2) that this assumption always holds when $r\leq 1$; to ease notation and terminology, we restrict to the hyperelliptic setting.

Fix a hyperelliptic curve $C/\QQ$ of genus $g \geq 3$, with a rational Weierstrass point $\infty$, and embed $C$ in its Jacobian $J$ via the Abel--Jacobi map $C \rightarrow J$ given by $\infty$.  Let $p$ be a prime of good reduction for $C$.  Let $W_d \coloneqq C+\dots + C \subset J$ be the image of all degree $d$ effective divisors, and let $\Lambda_C$ be as in Section \ref{subsec:vanishingintegrals}.  Define $J^{\Lambda_C}$ to be the kernel of pairing with elements of $\Lambda_C$, i.e. $$J^{\Lambda_C} \coloneqq \brk{D \in J(\Cp) \  \left|  \ \int_0^D \omega = 0, \ \forall \, \omega \in \Lambda_C\right.}.$$

Note that $J^{\Lambda_C}$ is also a $\Cp$-analytic manifold (in the sense of Bourbaki and Serre \cite[Chapter~III]{serre2009lie}), and in fact a $p$-adic Lie group.  
If we assume that $J(\mQ)$ has rank $\leq 1$, then $J^{\Lambda_C}$ has dimension 0 or 1 (as a manifold).  The next two lemmas use the topology on $J(\Cp)$ and $C(\Cp) \times \cdots \times C(\Cp)$ given by their structures as $\Cp$-analytic manifolds, unless otherwise noted.

\begin{lemma}\label{lem:isoladedpoints}
Let $0< d < g$ be integers.
%Let $d$ be a positive integer, and fix $g>d$ an integer. 
Assume that $J$ is geometrically simple and that $J(\mQ)$ has rank $\leq 1$.
Then $J^{\Lambda_C}\cap W_d$ consists of isolated points.
\end{lemma}

\begin{proof}
%We may assume $C$ has geometrically simple Jacobian $J$, by Proposition \ref{simple}. 
If $J^{\Lambda_C}$ is 0-dimensional, then we are done since $J^{\Lambda_C}$ is a closed subset of $J$.
 
If $J^{\Lambda_C}$ is 1-dimensional, let $P \in J^{\Lambda_C} \cap W_d$. We can choose a closed neighborhood $V$ of $P$ such that $V \cap J^{\Lambda_C}$ is diffeomorphic to a closed disk in $\Cp$ via \cite[Chapter~III, Section~3]{serre2009lie}. Since $W_d$ is a closed set, $V \cap J^{\Lambda_C} \cap W_d$ is given by the vanishing of a convergent 1-variable power series on this disk.  Thus, $V \cap J^{\Lambda_C} \cap W_d$ is either a finite set of points or all of $V\cap J^{\Lambda_C}$.

But in the latter case, note that $V\cap J^{\Lambda_C} \cap W_d$ is a translate of a closed disk centered at the origin, which makes it a translate of an infinite subgroup of $J^{\Lambda_C}$.  
Its Zariski closure would then be a translate of a positive-dimensional abelian subvariety of $J$ contained in $W_d$, but this contradicts our initial assumption that $J$ is geometrically simple. 
Therefore, $V\cap J^{\Lambda_C}\cap W_d$ is a finite set of points, so $P$ is isolated.
\end{proof}

Let $$(C^{d})^{\Lambda_C} \coloneqq \brk{(P_1,\dots,P_d) \in C(\mC_p)\times \dots \times C(\mC_p) \ \left|  \ F_{\omega}^d(P_1,\dots,P_d) =0, \, \forall \,\omega\in \Lambda_C \right. },$$
the inverse image of $J^{\Lambda_C}$ in $C^d$. For any subset $S \subset \Lambda_C$, let $(C^d)^S$ similarly denote pairs satisfying the condition for all $\omega \in S$. 

\begin{lemma}\label{isolatedincurve}
Let $0< d < g$ be integers. Assume that $J$ is geometrically simple and that $J(\mQ)$ has rank $\leq 1$.
Let $P = (P_1,\dots,P_d)$ be a point of $C^{d}(\mC_p)$ such that the divisor $P_1 + \cdots + P_d $ is non-special. 
Then $P$ is an isolated point of $(C^{d})^{\Lambda_C}$.
\end{lemma}
\begin{proof}
%In the proof of Proposition \ref{prop:finitenessdegreedpoints}, we showed that when $d$ is odd (resp.~even), the Abel--Jacobi map on $C^{(d)}$ has finite fibers away from the image of $C$ (resp.~away from $0$). 
The set of $d$-tuples in $C^{d}(\mC_p)$ which give special divisors is a closed subset. Away from this subset, the fibers of the map $C^d \rightarrow W_d$ are finite.
The result then follows from Lemma \ref{lem:isoladedpoints} and the fact that $C \times \cdots \times C$ is Hausdorff in its topology as a $\Cp$-analytic manifold.
\end{proof}

To conclude this section, we consider the locally affinoid structure of $(C^{d})^{\Lambda_C}$. 

\begin{defn}
For $P \in C(\Cp)$, $z_P$ a well-behaved uniformizer at $P$, and $m > 0$, let $$B_m(P,z_P) \coloneqq \brk{Q \in C(\Cp) \ \left| \ v(z_P(Q)) \geq m \right.}.$$
\end{defn}

Since $F_\omega^d$ has a convergent power series expression on the entire open polydisk $D_{\overline{P_1}} \times \cdots \times D_{\overline{P_d}}$, on any closed sub-polydisk it will actually give an element of that sub-polydisk's affinoid coordinate ring, which is a Tate algebra \cite[Section~7.1.1]{bosch1984non}.    
For any choices of $P_1, \dots ,P_d$ and $z_{P_1},\dots, z_{P_d}$ and $m > 0$, the set $$(C^d)^{\Lambda_C} \cap (B_m(P_1,z_{P_1}) \times \cdots \times B_m(P_d,z_{P_d}))$$ will have finitely many irreducible components as an affinoid subset of $B_m(P_1,z_{P_1}) \times \cdots \times B_m(P_d,z_{P_d})$, by \cite[Sect.~7.1.1, Cor.~8]{bosch1984non}. 
These components can be zero-dimensional or positive-dimensional.

\begin{lemma}\label{bestuniversalbasis}
Let $C/\QQ$ be a curve of genus $g \geq 2$, let $d$ be a positive integer, and let $p$ be a prime of good reduction for $C$. Suppose $C$ has rank $r \leq g-d$. Let $P_1, \dots, P_d \in C(\Cp)$, let $z_{P_i}$ be a well-behaved uniformizer at $P_i$ for $i =1,\dots, d$, and let $m > 0$. Then we can choose $\omega_1, \dots, \omega_d \in \Lambda_C$ as in Lemma \ref{universalbasis} such that furthermore the zero set $$(C^d)^{\{\omega_1,\dots, \omega_d\}} \cap (B_m(P_1,z_{P_1}) \times \dots \times B_m(P_d,z_{P_d}))$$ has the same positive-dimensional components as $$(C^d)^{\Lambda_C}\cap (B_m(P_1,z_{P_1}) \times \dots \times B_m(P_d,z_{P_d})).$$
\end{lemma}

\begin{proof}
The proof is similar to \cite[Proposition 5.7]{park16}, and proceeds by induction. \\ 

\noindent{\em Claim}: For $k=1,\dots,d$, we can choose $\omega_1,\dots, \omega_k$ as in Lemma \ref{universalbasis} such that the set of components of codimension at most $k-1$ for $$(C^d)^{\{\omega_1,\dots, \omega_k\}} \cap (B_m(P_1,z_{P_1}) \times \dots \times B_m(P_d,z_{P_d}))$$
and
$$(C^d)^{\Lambda_C} \cap (B_m(P_1,z_{P_1}) \times \dots \times B_m(P_d,z_{P_d}))$$
are the same.\\

\begin{proof}[Proof of Claim]
The base case $k=1$ is trivial.  For the induction step, suppose it holds for a given value of $k$, for forms $\omega_1,\dots,\omega_k$. Choose $\omega_{k+1}$ linearly independent from $\omega_1,\dots, \omega_k$, and as in Lemma \ref{universalform}, such that 
$$(C^d)^{\{\omega_1,\dots, \omega_{k+1}\}} \cap (B_m(P_1,z_{P_1}) \times \dots \times B_m(P_d,z_{P_d}))$$
has the minimal number of 
 codimension $k$ components $V_1,\dots,V_s$ for any such choice of $\omega_{k+1}$.
 
 Suppose the conclusion of the claim is false for $k+1$.  Then without loss of generality, we may assume that $V_s$ is not a component of $(C^d)^{\Lambda_C} \cap (B_m(P_1,z_{P_1}) \times \dots \times B_m(P_d,z_{P_d}))$.  Let $V_{s+1}, \dots, V_t$ be any codimension $k$ components of $(C^d)^{\{\omega_1,\dots, \omega_k\}} \cap (B_m(P_1,z_{P_1}) \times \dots \times B_m(P_d,z_{P_d}))$ that are distinct from each $V_1, \dots, V_s$.

Since $V_s$ is not a component of $(C^d)^{\Lambda_C} \cap (B_m(P_1,z_{P_1}) \times \dots \times B_m(P_d,z_{P_d}))$, we may choose a normalized form $\omega' \in \Lambda_C$ such that $F_{\omega'}$ does not vanish at some point $R_s$ of $V_s$.  Then for any integer $n$, we have that $F^d_{\omega_{k+1}+p^n\omega'}$ does not vanish identically on $V_s$, since $F^d_{\omega_{k+1}}$ does but $F^d_{p^n\omega'}=p^nF^d_{\omega'}$ does not.

For each $i=s+1, \dots, t$, we can choose a point $R_i$ on $V_i$ at which $F^d_{\omega_{k+1}}$ does not vanish.  Choose $n$ to be a sufficiently large positive integer such that $v(F_{p^n\omega'}(R_i)) = n + v(F_{\omega'}(R_i)) > v(F_{\omega_{k+1}}(R_i))$ for each $i=s+1, \dots, t$. Then $F_{\omega_{k+1} + p^n\omega'}$ does not vanish identically on any of $V_{s+1}, \dots, V_t$.

Note that $\omega_{k+1}+p^n\omega'$ has the same reduction as $\omega_{k+1}$, and thus the conclusion of Lemma \ref{universalbasis} still holds. Also, it is linearly independent from $\omega_1,\dots, \omega_k$, since $F^d_{\omega_1},\dots,F^d_{\omega_k}$ vanish at $R_s$ and our integration pairing is $\Qp$-linear.  But by construction, the codimension $k$ components of $(C^d)^{\{\omega_1,\dots, \omega_d, \omega_{k+1}+p^n\omega'\}} \cap (B_m(P_1,z_{P_1}) \times \dots \times B_m(P_d,z_{P_d}))$ are contained in $\{V_1, \dots, V_{s-1}\}$, which contradicts the minimality of $\omega_{k+1}$.
\end{proof}
%\noindent{\em Proof of claim}: 
The lemma is the $k=d$ case of the claim.
\end{proof}

\section{Bounding the number of unexpected degree $d$ points}\label{sec:proofmain}
In this section, we prove the first two statements of Theorem \ref{thm:degreedbound}. We refer the reader to Subsection \ref{subsec:term} for the definition of unexpected degree $d$ points.
%In this section, we discuss the intersection locus of the tropicalization of the locally analytic functions defined in Sections \ref{sec:padicprem} and \ref{sec:analyticloci}. 

Let $(P_1,\dots,P_d)$ be a conjugate $d$-tuple of degree $d$ points.

\begin{lemma}\label{lem:rankonelemmadegreed}
Let $d>1$ be a positive integer, let $C/\QQ$ be a hyperelliptic curve of genus $g > d$, with a rational Weierstrass point, geometrically simple Jacobian with $r \leq 1$, and let $p$ be an odd prime of good reduction for $C$.  
%For each partition $\rho = (\rho_1,\rho_2,\dots,\rho_n)$ of $d$, 
Let $P_{1},\cdots ,P_d\in C(\Qbar)$ be a conjugate $d$-tuple with well-behaved uniformizers $z_{P_{1}},\dots , z_{P_{d}}$. 
Let $(Q_1, \dots , Q_d)$ be a $d$-tuple of unexpected conjugate degree $d$ points with the same reduction as $(P_1,\dots ,P_d)$ modulo $p$.  

Then $\{(Q_1,\dots,Q_d)\}$ is a zero-dimensional component of $$(C^d)^{\Lambda_C}\cap (B_{\frac{1}{d^2}}(P_{1},z_{P_{1}}) \times  \cdots \times  B_{\frac{1}{d^2}}(P_{d},z_{P_{d}})).$$
\end{lemma}

\begin{proof}
By Remark \ref{valuationbound}, $\{(Q_1,\dots , Q_d)\} $ is contained in $$B_{\frac{1}{d^2}}(P_{1},z_{P_{1}}) \times  \cdots \times  B_{\frac{1}{d^2}}(P_{d},z_{P_{d}}). $$  By Lemma \ref{isolatedincurve}, we have that $\{(Q_1,\dots , Q_d)\}$ is a zero-dimensional component of $$(C^d)^{\Lambda_C}\cap (B_{\frac{1}{d^2}}(P_{1},z_{P_{1}}) \times  \cdots \times  B_{\frac{1}{d^2}}(P_{d},z_{P_{d}})).$$
\end{proof}

In light of Lemma~\ref{lem:rankonelemmadegreed}, we work to bound the number of these zero-dimensional components. 
We use the convention that $v(0)=\infty$.  For a positive rational number $m$, let $D_m = \{\alpha \in \Cp \ | \ v(\alpha) \geq m\}$, let $D^d_m$ denote its $d$-fold product, and let $\Cp\langle D^d_m\rangle$ denote the Tate algebra of functions in its affinoid coordinate ring. 
The key tool we use to control zero-dimensional components comes from Park \cite{park16} and builds off results of Rabinoff \cite{rabinoff2012tropical} in tropical deformation and tropical intersection theory.

Before we can state the theorem, we recall some definitions. 
%First, we define the tropicalization of an element in $\Cp\langle D^d_m\rangle$. 

\begin{defn}
For $F(t_1,\dots,t_d) = \sum_{u \in \mZ_{\geq 0}^d} a_ut^u \in \Cp\langle D^d_m\rangle$, we define the \textit{tropicalization of $F$} as 
\[
\Trop(F) := \overline{\brk{(v(z_1),\dots,v(z_d)) : F(z_1,\dots,z_d) = 0, (z_1,\dots,z_d) \in D_m^d)}},
\]
where we take the topological closure in $\mR^d$.
\end{defn}

Next, for a set $S \subset \RR^d$, let $\conv(S)$ denote its convex hull.
A necessary condition for a series to sum to $0$ in a non-Archimedean field is that it has a term of minimal valuation, and that this term is not unique.  In the below definition, the relevant $(w_1, \dots , w_d)$ can thus be thought of as candidates for the coordinate-wise valuations of zeros for the power series $F$, where we only look for zeros whose coordinates have valuation at least $m$, and 
the $(u_1, \dots , u_d)$ are the multi-indices of terms that could have minimal valuation after plugging in such a zero. 
\begin{defn}
Let $m$ be a positive rational number, and let $F(t_1,\dots,t_d) = \sum_{u \in \mZ_{\geq 0}^d} a_ut^u \in \Cp\langle D^d_m\rangle$.  We define the \textit{Newton polygon} of $F$ (with respect to $m$) to be the set $\New_m(F) \subset \RR^d$ given by 
\begin{align*}
\New_m(F) := \conv(\{& u =(u_1, \dots, u_d) \in  \mZ_{\geq 0}^d \ | \  \exists (w_1, \dots, w_d) \in \QQ^d \text{ with each } w_i \geq m \text{ s.t.}\\
& \exists \, u' \in \mZ_{\geq 0}^d \text{ with } u \neq u' \text{ and } v(a_u) + \sum_i w_iu_i = v(a_{u'}) + \sum_i w_iu'_i,\\
&  \text{and } \forall\, u''  \in \mZ_{\geq 0}^d, v(a_u) + \sum_i w_iu_i \leq v(a_{u''}) + \sum_i w_iu''_i \}).
\end{align*}
\end{defn}

A \textit{polytope} is the convex hull of finitely many points of a Euclidean space.  We need to define the mixed volume of a collection of polytopes in a given Euclidean space.  For polytopes $Z_1, \dots, Z_d \subset \RR^d$ and positive real numbers $\lambda_1, \dots, \lambda_d$, the volume of the scaled Minkowski sum $\lambda_1Z_1 + \dots + \lambda_dZ_d = \{\lambda_1z_1 + \dots + \lambda_dz_d  \mid  z_i \in Z_i\}$ is known to be given by a homogeneous polynomial of degree $d$ in the coefficients $\lambda_1, \dots, \lambda_d$.  The \textit{mixed volume} of the polytopes, denoted $\MV(Z_1, \dots, Z_d)$, is defined to be the coefficient of $\lambda_1\lambda_2\cdots\lambda_d$ in that polynomial.

\begin{exam}\label{mvexample}
For $a \geq 1$ and $d$ a positive integer, let $$Z_1 = \dots = Z_d = \conv(\{e_1, \dots, e_d, ae_1, \dots, ae_d\}) \subset \RR^d,$$
where $e_i$ is the vector with a 1 in the $i$-th place and $0$'s elsewhere. Then $$\lambda_1Z_1 + \cdots +\lambda_dZ_d = \conv((\lambda_1+\dots +\lambda_d)\{e_1, \dots, e_d, ae_1, \dots, ae_d\}),$$ and a quick calculation gives $\MV(Z_1,\dots,Z_d) = a^d-1$.
\end{exam}

\begin{theorem}[\protect{\cite[Theorem~4.18 \& Proposition 5.7]{park16}}]\label{parkthm}
Let $m$ be a positive rational number, and let $F_1, \dots, F_d \in \Cp\langle D^d_m\rangle$ be power series such that for any $i, j \in \{1, \dots, d\}$, $F_i$ has a term of the form $c t_j^N$, with $c \neq 0$ and $N>0$.

Suppose that the tropicalization of an isolated point in the intersection of the zero loci of the $F_i$ is isolated in the intersection of the tropicalizations $\Trop(F_i)$. 
Then the number of zero-dimensional components of the zero locus $V(F_1, \dots, F_d)$ in $D^d_m \cap (\Cp^\times)^d$ is at most $\MV(\New_m(F_1), \dots, \New_m(F_d)).$
\end{theorem}

\begin{remark}\label{remark:assdiscussion}
The condition that the tropicalization of an isolated point in the intersection of the zero loci of the $F_i$ is isolated in the intersection of the tropicalizations does not appear in \cite{park16}, but does appear in the previous theorem from \cite{rabinoff2012tropical} that Park uses. It seems to be necessary if one wishes to use the mixed volume $\MV(\New_m(F_1), \dots, \New_m(F_d))$ to bound the number of zero-dimensional components of $V(F_1, \dots, F_d)$ in $D^d_m \cap (\Cp^\times)^d$.

If the tropicalization of an isolated point in the intersection of the zero loci of the $F_i$ lands in a  positive-dimensional component of the intersection of their tropicalizations, then one could hope to use work of Osserman--Rabinoff \cite{osserman2011lifting} to bound the number of zero-dimensional components in terms of the \textit{stable tropical intersection}; see \cite{Jensen2016} for definitions and results concerning stable tropical intersections. 
However, one cannot use mixed volumes of Newton polygons to compute these stable intersection numbers.
\end{remark}

Since we use Theorem \ref{parkthm} to explicitly bound the number zero-dimensional components from Lemma \ref{lem:rankonelemmadegreed}, we include an assumption.

\begin{ass}\label{ass:Park}
For each $g\geq 2$, more than $75\%$ of hyperelliptic curves over $\mQ$ of genus $g$ with a rational Weierstrass point satisfy the following condition:
\begin{equation}
\begin{array}{l}
\text{for each positive rational number $m$ and each residue disk } D_m^{d}, \text{there exist $d$ linearly} \\
\text{independent $1$-forms $\omega_1,\dots, \omega_d$ as in Lemma \ref{bestuniversalbasis}, such that the tropicalization of any} \\
\text{isolated point in the intersection of the zero loci of the } F_{\omega_i}^d \text{ is isolated in the intersection} \\
\text{of the tropicalizations} \Trop(F_{\omega_i}^d).
\end{array}
 \tag{\(\dagger\)}
\end{equation} 
Furthermore, the same holds if one averages within a family defined by a finite set of congruence conditions.
\end{ass}

% to the finitely many degree $d$ points described in Theorem \ref{thm:degreedbound} as \textit{unexpected} degree $d$ points, without mention to whether $d$ is even or odd. 

Returning to the proof of the first two statements of Theorem \ref{thm:degreedbound}, we use Theorem \ref{parkthm} to bound the maximal number of zero-dimensional components mentioned in Lemma \ref{lem:rankonelemmadegreed}.

\begin{lemma}\label{lem:Fpreductionbound}
Let $d>1$ be a positive integer and let $C/\QQ$ be a hyperelliptic curve of genus $g > d$, with a rational Weierstrass point, geometrically simple Jacobian with $r \leq 1$, good reduction at a prime $p> d^2 + 3$, and which satisfies condition $(\dagger)$. 

%For each partition $\rho = (\rho_1,\rho_2,\dots,\rho_d)$ of $d$, let $(P_{1,1},\dots,P_{1,\rho_1}, \dots,P_{n,1},\dots,P_{n,\rho_n}) \in C(\Qbar)$ be the $d$-tuple prescribed above, with $(\overline{P_{1,1}},\dots,\overline{P_{1,\rho_1}}, \dots,\overline{P_{n,1}},\dots,\overline{P_{n,\rho_n}})  \in C_{\mF_p}(\mF_{p^d})$. 

Let $P_{1},\cdots ,P_d \in C^d(\Qbar)$ be a conjugate $d$-tuple with well-behaved uniformizers $z_{P_{1}},\dots , z_{P_{d}}$. 
%Let $(Q_1, \dots , Q_d)$ be a $d$-tuple of unexpected conjugate degree $d$ points with the same reduction as $(P_1,\dots ,P_d)$ modulo $p$.  
Then there are at most $3^d$ ordered tuples of unexpected conjugate degree $d$ points in $D_{\overline{P_{1}}} \times \cdots \times D_{\overline{P_{d}}}$, i.e.~with the same reduction as $(P_1,\dots ,P_d)$ modulo $p$.  
\end{lemma}

\begin{proof}
%Fix any partition $\rho = (\rho_1,\rho_2,\dots,\rho_d)$; we shall abuse notation and refer to the $d$-tuple corresponding to $\rho$ as $P_1,\dots,P_d$. 
Using Lemmas \ref{universalbasis} and \ref{bestuniversalbasis}, we can choose linearly independent, normalized forms $\omega_1, \dots,\omega_d \in \Lambda_C$ such that $n(\omega_i,\overline{P_j}) = n(\Lambda_C, \overline{P_j})$ for $i,j \in \{1,\dots,d\}$.  By Theorem \ref{stollbound}, these numbers are at most 2. 

Viewed as a function on $D^d$ via $z_{P_1},\dots,z_{P_d}$, $F_{\omega_1}^d$ is given by 
$$\sum_{i=0}^\infty \frac{a_{1,i}}{i+1}t_1^{i+1} + \cdots + \sum_{i=0}^\infty \frac{a_{d,i}}{i+1}t_d^{i+1}.$$
For each $i\in \brk{1,\dots,d}$, we have that $v(a_{i,j}) =0$ for some $j\leq 2$. Similar statements holds for $F_{\omega_2}^d,\dots,F_{\omega_d}^d$. 

By a standard Newton polygon argument (cf.~proof of \cite[Lemma 6.1 \& Proposition 6.3]{stoll2006independence} using our assumption that $p>d^2 + 3$) applied to each of these sums, we see that $\New_{1/d^2}(F_{\omega_1}^d),\New_{1/d^2}(F_{\omega_2}^d),\dots,$ and $\New_{1/d^2}(F_{\omega_d}^d)$ are contained in the set $$\conv(\{(1,0,\dots,0), (3,0,\dots,0), (0,1,\dots,0), (0,3,\dots,0),\dots,(0,0,\dots,1),(0,0,\dots,3)\}) \subset \RR^d.$$  By Theorem \ref{parkthm} and Example \ref{mvexample}, there are at most $3^d-1$ zero-dimensional components of interest away from $(P_1, \dots,P_d)$. Since $(P_{1}, \dots,P_{d})$ could be a $d$-tuple of unexpected conjugate degree $d$ points, we have at most $3^d$ zero-dimensional components in $D^d$. We are now done by Lemmas \ref{bestuniversalbasis} and \ref{lem:rankonelemmadegreed}.
\end{proof}

\begin{proof}[Proof of Theorem \ref{thm:degreedbound}]
%Let $d>1$ be a positive integer. 
Choose a prime $p>d^2 + 3$. Among all hyperelliptic curves of genus $g$ with a rational Weierstrass point, those with good reduction at $p$ are defined by finitely many congruence conditions on their minimal equations, and thus constitute a positive proportion of all such curves.
Proposition \ref{simple} tells us that asymptotically, $100\%$ of curves in this subfamily have geometrically simple Jacobian, and by Corollary \ref{goodrankcor}, at least 25\% of these curves also have Jacobian rank $r \leq 1$. 
Furthermore, under Assumption \ref{ass:Park}, a positive proportion of these curves satisfy condition $(\dagger)$. 

Let $C$ be such a curve of genus $g$. 
Given a $d$-tuple $(Q_1,\dots,Q_d)$ of conjugate degree $d$ points, the reduction of each $Q_i$ is certainly contained in $C_{\mF_p}(\mF_{p^{m_i}})$ for some $1\leq m_i \leq d$. 
Note that since $C$ is an odd hyperelliptic curve with good reduction at $p$, the size of $C_{\mF_p}(\mF_{p^{m_i}})$ is less than or equal to $2p^{m_i} + 1$. 
Crudely, there are at most $(d\cdot (2p^{d} + 1))^d$ possible reductions for $(Q_1,\dots,Q_d)$ modulo $p$.
By Lemma \ref{lem:Fpreductionbound}, there are at most $3^d$ ordered tuples of unexpected conjugate degree $d$ points with each reduction. 
Thus, we may take $$B_d = (3d\cdot (2p^{d} + 1))^d \leq (3d\cdot (2(2d^2+3)^{d} + 1))^d,$$ by Bertrand's postulate \cite[Tome I, p. 64]{chebyshev}.
\end{proof}

\section{Explicit bounds on the number of unexpected quadratic points}\label{boundingquadratic}
In this section, we prove Theorem \ref{hyperthm}. 

\begin{lemma}\label{rankonelemma}
Let $C/\QQ$ be a hyperelliptic curve of genus $g \geq 3$, with a rational Weierstrass point, geometrically simple Jacobian with $r \leq 1$, and let $p$ be an odd prime of good reduction for $C$. 

Let $P_1, P_2 \in C(\Qbar)$ be either two rational points or a pair of conjugate quadratic points, with well-behaved uniformizers $z_{P_1}, z_{P_2}$.  Let $(Q_1, Q_2)$ be a pair of unexpected conjugate quadratic points with the same reduction as $(P_1,P_2)$.  Then $\{(Q_1,Q_2)\}$ is a zero-dimensional component of $(C^2)^{\Lambda_C}\cap (B_{\frac{1}{2}}(P_1,z_{P_1}) \times B_{\frac{1}{2}}(P_2,z_{P_2}))$.
\end{lemma}

\begin{proof}
The field $\Qp(P_1,P_2,Q_1,Q_2)$ is certainly contained in the compositum of the three quadratic extensions of $\Qp$.  Since $p$ is odd, that compositum has ramification degree $e=2$ over $\Qp$. 
The result now follows from Lemma \ref{lem:rankonelemmadegreed}. 
%so $\{(Q_1,Q_2)\} \subset B_{\frac{1}{2}}(P_1,z_{P_1}) \times B_{\frac{1}{2}}(P_2,z_{P_2})$, by Remark \ref{valuationbound}.  By Lemma \ref{isolatedincurve}, it is a zero-dimensional component of $(C^2)^{\Lambda_C}\cap (B_{\frac{1}{2}}(P_1,z_{P_1}) \times B_{\frac{1}{2}}(P_2,z_{P_2}))$.
\end{proof}

\begin{lemma} \label{newtonbounds}
Suppose $p \neq 2$ or $5$, and $\sum_{i=0}^{\infty} a_it^i \in \Cp\fps{t}$ is a power series with integral coefficients. If $v(a_0) =0$, then for $f(t)=\sum_{i=0}^\infty \frac{a_i}{i+1}t^{i+1}$, we have $\New_1(f) =\emptyset$ and $\New_{1/2}(f) \subset [1,3]$.  If $v(a_1) =0$ or $v(a_2)=0$, then $\New_1(f), \New_{1/2}(f) \subset [1,3]$.
\end{lemma}

\begin{proof}
We begin with the case where $v(a_0)=0$.  Then for any $w \geq 1$ and $i > 0$, we have $v(\frac{a_0}{1}) + 1\cdot w = w < v(\frac{a_i}{i+1}) + (i+1)\cdot w$. To see this, note that the right-hand side is no smaller than $-v(i+1) + (i+1)w$, and $v(i+1)$ is strictly less than $i$ for $i>0$ and $p > 2$.

For $w \geq \frac{1}{2}$, things can be different: for example, suppose $p=3$, $v(a_2)=0$, and $w=\frac{1}{2}$.  Then $\frac{1}{2}=v(\frac{a_2}{3}) + 3\cdot w= v(\frac{a_0}{1}) + 1\cdot w$.  But past $i=2$, strict inequality holds, so we have $\New_{1/2}(f) \subset [1,3]$.  The rest of the cases proceed similarly. 
\end{proof}

\begin{lemma}\label{Fpreductionbound}
Let $C/\QQ$ be a hyperelliptic curve of genus $g \geq 3$, with a rational Weierstrass point, geometrically simple Jacobian with $r \leq 1$, good reduction at 3, which satisfies condition $(\dagger)$. 

Let $P_1, P_2 \in C(\Qbar)$ be either two rational points or a pair of conjugate quadratic points, with $\overline{P_1}, \overline{P_2} \in C_{\mathbb{F}_3}(\mathbb{F}_3)$. Then there are at most $8$ ordered pairs $(Q_1, Q_2)$ of unexpected conjugate quadratic points in $D_{\overline{P_1}} \times D_{\overline{P_2}}$ that are not equal to $(P_1, P_2)$.
\end{lemma}

\begin{proof}
Using Lemmas \ref{universalbasis} and \ref{bestuniversalbasis}, choose linearly independent, normalized forms $\omega_1, \omega_2 \in \Lambda_C$ such that $n(\omega_i,\overline{P_j}) = n(\Lambda_C, \overline{P_j})$ for $i,j \in \{1,2\}$.  By Theorem \ref{stollbound}, these numbers are at most 2.  
%Expand $\omega_1$ with respect to well-behaved uniformizers $z_{P_1}$ and $z_{P_2}$ as $(\sum_{i=0}^\infty a_iz_{P_1}^i)dz_{P_1}$ and $(\sum_{i=0}^\infty b_iz_{P_2}^i)dz_{P_2}$.  
Viewed as a function on $D^2$ via $z_{P_1}$ and $z_{P_2}$, $F_{\omega_1}^2$ is given by 
$$\sum_{i=0}^\infty \frac{a_i}{i+1}t_1^{i+1} + \sum_{i=0}^\infty \frac{b_i}{i+1}t_2^{i+1}.$$
A similar statement holds for $F_{\omega_2}^2$.

By construction of $\omega_1$ (and similarly for $\omega_2$), we have $v(a_i) =0$ for some $i\leq 2$, and $v(b_j)=0$ for some $j \leq 2$.  Now by Lemma \ref{newtonbounds} applied to each of these two sums, we see that both $\New_{1/2}(F_{\omega_1}^2)$ and $\New_{1/2}(F_{\omega_2}^2)$ are contained in the set $$\conv(\{(1,0), (3,0), (0,1), (0,3)\}) \subset \RR^2.$$  By Theorem \ref{parkthm} and Example \ref{mvexample}, there are at most $3^2-1 =8$ zero-dimensional components of interest away from $(P_1, P_2)$.  We are done by Lemmas \ref{bestuniversalbasis} and \ref{rankonelemma}.
\end{proof}

\begin{lemma}\label{Fptworeductionbound}
Let $C/\QQ$ be a hyperelliptic curve of genus $g \geq 3$, with a rational Weierstrass point, geometrically simple Jacobian with $r \leq 1$,  good reduction at 3, which satisfies condition $(\dagger)$. 

Let $P_1, P_2 \in C(\Qbar)$ be a pair of conjugate quadratic points, with $\overline{P_1}, \overline{P_2} \in C_{\mathbb{F}_3}(\mathbb{F}_9) \setminus C_{\mathbb{F}_3}(\mathbb{F}_3)$.  If $n(\Lambda_C, \overline{P_1}) =1$, there are at most $8$ pairs $(Q_1, Q_2)$ of unexpected conjugate quadratic points in $D_{\overline{P_1}} \times D_{\overline{P_2}}$ that are not equal to $(P_1, P_2)$.  If $n(\Lambda_C, \overline{P_1}) = 0$, there are no such pairs.
\end{lemma}

\begin{proof}
The proof is similar to that of Lemma \ref{Fpreductionbound}, with two changes.  First, one can consider $\New_1$ instead of $\New_{1/2}$, using Remark \ref{valuationbound} and the fact that $\QQ_3(P_1,Q_1)$ is unramified.  Second, note that since $P_1$ and $P_2$ reduce to (necessarily distinct) points outside of $C_{\mathbb{F}_3}(\mathbb{F}_3)$, we know that $P_1$ and $P_2$ remain conjugate over $\QQ_3$.  Thus their reductions are conjugate over $\mathbb{F}_3$, so we have $n(\Lambda_C, \overline{P_1}) = n(\Lambda_C, \overline{P_2})$.
%since $n(\Lambda_C, \overline{Q_i})$ is defined using forms over $\Qp$.  
By Theorem \ref{stollbound}, this common value can only be 0 or 1.
\end{proof}

\begin{lemma}\label{goodfamily}
For each $g\geq 3$, there exists a congruence family of genus $g$ hyperelliptic curves with a rational Weierstrass point, such that any curve $C$ in the family has good reduction at 3, and satisfies $C_{\mathbb{F}_3}(\mathbb{F}_3)=\{\overline{\infty}\}$ and $C_{\mathbb{F}_3}(\mathbb{F}_9) = \{\overline{\infty}, (0, \pm \alpha), (1, \pm \alpha), (2, \pm \alpha)\}$, with $\alpha \in (\mathbb{F}_9 \setminus \mathbb{F}_3)$.
\end{lemma}

\begin{proof}
For each $g$, consider the families of hyperelliptic curves whose reduction mod 3 is given by:
\begin{align*}
y^2 &= f_g(x) = x^{2g+1}+2x^{9}+2 & \text{ for } g\equiv 1 \mod 4,& \\
y^2 &= f_g(x) = x^{2g+1}+2x^{15}+2 &\text{ for } g\equiv 2 \mod 4, &\\
y^2 &= f_g(x) = x^{2g+1}+2x^{5}+2 & \text{ for } g\equiv 3 \mod 4, &\\
y^2 &= f_g(x) = x^{2g+1} + x^3 + x + 2 & \text{ for } g\equiv 0 \mod 4 \text{  and  } g \equiv 0,\,1 \mod 3, &\\
y^2 &= f_g(x) = x^{2g+1} + x^9 + x^3 +  2 & \text{ for } g\equiv 0 \mod 4 \text{  and  } g \equiv 2\mod 3 &.
\end{align*}
For $g=3$, 
%we take the family of curves whose reduction mod 3 is given by $y^2 = f(x) = x^7+2x^5+2$.  
one checks directly that $f_g(x)$ represents as few squares as is possible: $f(\mathbb{F}_3) = \{2\}$, and $f(\mathbb{F}_9 \setminus \mathbb{F}_3) \subset (\mathbb{F}_9 \setminus \mathbb{F}_9^2)$.  This is equivalent to the lemma's condition on $\mathbb{F}_3$- and $\mathbb{F}_9$-points. 
In general, for $g \equiv 3\mod 4$, note that $f_g(x)=x^{2g+1}+2x^5+2$ defines the exact same function as $x^7+2x^5+2$ on $\mathbb{F}_9$, as $x^k $ and $x^{k + 8r}$ define the same function on $\mF_{9}$ for natural numbers $k,\, r$. 
The same argument works for the other values of $g$.

%For $g \equiv 1, 2\mod 4$, we can take curves whose reduction mod 3 is given by $y^2=f_g(x)$, with $f_g(x)= x^{2g+1}+2x +2$, or $x^{2g+1}+2x^7+2$, respectively. For $g\equiv 0 \mod 4$, we have two cases depending on the congruence class of $n = 2g+1$ mod 3: for $n\equiv 0,1 \mod 3$,  $f_g(x) = x^{2g+1}+x^3+x+2,$ and for $n\equiv 2\mod 3$, $f_g(x) = x^{2g+1} + x^9  + x^3 + 2$.

We conclude by showing that these polynomials are square-free over $\mF_3$. Over any field, the discriminant of a trinomial $x^n + ax^k + b$ is $$(-1)^{\frac{n(n-1)}{2}}b^{k-1}[n^{n_1}b^{n_1-k_1}+(-1)^{n_1+1}(n-k)^{n_1-k_1}k^{k_1}a^{n_1}]^d,$$
where $d=(n,k)$ and $n = n_1d, k=k_1d$ \cite[Theorem 2]{swan}.  Thus for $g \equiv 1, 2, 3\mod 4$, the discriminant of $f_g$ in $\mF_3$ is $$\pm [(2g+1)+(2g+1-k)^2k\cdot 2],$$ with $k=1, 7,$ or $5$, respectively.  This is non-zero in $\mF_3$ for any value of $g$.

For $g \equiv 0\mod 4$, we use a different method. It suffices to show $f_g(x)$ and $f_g'(x)$ have no common factors.  If $g \equiv 1\mod 3$, then $f_g'(x)=1$, so this is clear. If $g \equiv 2\mod 3$, then $f'_g(x) = -x^{2g}$, and so this also is clear. Lastly if $g \equiv 0\mod 3$, then $f_g'(x) = x^{2g}+1$, so any common factor would divide $f_g(x)-xf_g'(x)=x^3+2=(x+2)^3$.  Thus we see $f_g$ and $f_g'$ are coprime.  
\end{proof}

\begin{proof}[Proof of Theorem \ref{hyperthm}]
For a given value of $g$, the family given by Proposition \ref{simple} and Lemma \ref{goodfamily} comprises a positive proportion of all hyperelliptic curves with a rational Weierstrass point, since the latter is defined by finitely many congruence conditions.  By Corollary \ref{goodrankcor}, at least 25\% of the curves in this family have rank $r \leq 1$; this is still a positive proportion of all the curves.  
Furthermore, under Assumption \ref{ass:Park}, a positive proportion of these curves satisfy condition $(\dagger)$. 
%When $g=3$, this proportion does not change if we restrict to non-bielliptic curves, by Lemma \ref{negligibleelliptic}.

Let $C$ be such a curve.  Any pair of conjugate quadratic points on $C$ that reduce to $\mathbb{F}_3$-points will have to lie in $D_{\overline{\infty}} \times D_{\overline{\infty}}$.  We may choose $P_1 = P_2 = \infty$ and apply Lemma \ref{Fpreductionbound} to conclude there are at most eight ordered pairs $(Q_1, Q_2)$ of unexpected conjugate quadratic points in this residue class.  But since $\overline{Q_1}=\overline{Q_2}$, each \textit{unordered} pair is counted twice, so there are at most four unordered pairs of such quadratic points.

If the minimal Weierstrass model of $C$ is $y^2=f(x)$, note that (for some choice of square root) the pair of \textit{expected} quadratic points $(i, \pm \sqrt{f(i)})$ reduces to $(\overline{i}, \pm \alpha)$ for each $\overline{i}=0,1,2.$  In $D_{(\overline{i},\alpha)} \times D_{(\overline{i},-\alpha)}$, we may choose $P_1 = (i, \sqrt{f(i)}), P_2 = (i, -\sqrt{f(i)})$, and apply Lemma \ref{Fptworeductionbound}.  If $n(\Lambda_C, (\overline{i},\alpha)) = 0$, we conclude there are no unexpected pairs in this residue class.  If $n(\Lambda_C, (\overline{i},\alpha)) = 1$, there are at most eight.

By Theorem \ref{stollbound}, at most one value of $\overline{i}=0,1,2$ will have $n(\Lambda_C, (\overline{i},\alpha)) = n(\Lambda_C, (\overline{i},-\alpha))= 1$, and all the others will be 0.  Thus, there are at most $4 + 8 = 12$ unordered pairs of unexpected conjugate quadratic points.
\end{proof}

\section{Explicit bounds on the number of cubic points}\label{boundingcubic}
In this section, we prove Theorem \ref{cubicthm}. 
%As in Section \ref{boundingquadratic}, we work to bound the number of zero-dimensional components of $(C^3)^{\Lambda_C}\cap (B_m(P_1,z_{P_1}) \times B_m(P_2,z_{P_2}) \times B_m(P_3,z_{P_3}))$ for a triple of conjugate cubic points $(P_1,P_2,P_3)$ and $m\geq 1/3$.  

\begin{lemma} \label{newtonboundscubic}
Suppose $\sum_{i=0}^{\infty} a_it^i \in \mC_3\fps{t}$ is a power series with integral coefficients. If $v(a_0) = 0$, $v(a_1) =0$, or $v(a_2)=0$, then for $f(t)=\sum_{i=0}^\infty \frac{a_i}{i+1}t^{i+1}$, we have $\New_{1/3}(f) \subset [1,3]$. 
\end{lemma}

\begin{proof}
We proceed as in Lemma \ref{newtonbounds}. If $v(a_0) = 0$, then for $w\geq 1/3$ and $i>2$, we have that $v(\frac{a_0}{1}) + 1\cdot w = w < v(\frac{a_i}{i+1}) + (i+1)\cdot w$. Recall that the right-hand side is no smaller than $-v(i+1) + (i+1)w$, and so it suffices to prove that for $i>2$, $w < -v(i+1) + (i+1)w$. A short induction argument on $i$ using the non-Archimedean properties of $v$ and taking $w = 1/3$ yields the desired result. The remaining cases follow in a similar fashion. 
\end{proof}

\begin{lemma}\label{Fpreductionboundcubic}
Let $C/\QQ$ be a hyperelliptic curve of genus $g \geq 4$, with a rational Weierstrass point, geometrically simple Jacobian with $r \leq 1$, good reduction at 3, which satisfies condition $(\dagger)$.  

Let $P_1, P_2, P_3 \in C(\QQ)$ be three rational points. Then there are at most $26$ ordered triples $(Q_1, Q_2,Q_3)$ of conjugate cubic points in $D_{\overline{P_1}} \times D_{\overline{P_2}} \times D_{\overline{P_3}}$. 
\end{lemma}

\begin{proof}
As in Lemma \ref{Fpreductionbound}, we use Lemmas \ref{universalbasis} and \ref{bestuniversalbasis} to choose linearly independent, normalized forms $\omega_1, \omega_2, \omega_3 \in \Lambda_C$ such that $n(\omega_i,\overline{P_j}) = n(\Lambda_C, \overline{P_j})  $ for $i,j \in \{1,2,3\}$.  By Theorem \ref{stollbound}, these numbers are at most 2.  
%Expand $\omega_1$ with respect to well-behaved uniformizers $z_{P_1}$ and $z_{P_2}$ as $(\sum_{i=0}^\infty a_iz_{P_1}^i)dz_{P_1}$ and $(\sum_{i=0}^\infty b_iz_{P_2}^i)dz_{P_2}$.  
Viewed as a function on $D^3$ via $z_{P_1}$, $z_{P_2}$, and $z_{P_3}$, $F_{\omega_1}^3$ is given by 
$$\sum_{i=0}^\infty \frac{a_i}{i+1}t_1^{i+1} + \sum_{i=0}^\infty \frac{b_i}{i+1}t_2^{i+1} + \sum_{i=0}^\infty \frac{c_i}{i+1}t_3^{i+1}.$$
A similar statement holds for $F_{\omega_2}^3$ and $F_{\omega_3}^3$.

By construction of $\omega_1$ (and similarly for $\omega_2$ and $\omega_3$), we have $v(a_i) =0$ for some $i\leq 2$, $v(b_j)=0$ for some $j \leq 2$, and $v(c_k) = 0$ for some $k\leq 2$.  Now by Lemma \ref{newtonboundscubic} applied to each of these three sums, we see that $\New_{1/3}(F_{\omega_1}^3)$, $\New_{1/3}(F_{\omega_2}^3)$, and $\New_{1/3}(F_{\omega_3}^3)$ are contained in the set $$\conv(\{(1,0,0),(3,0,0),(0,1,0,),(0,3,0), (0,0,1), (0,0,3)\}) \subset \RR^3.$$  By Theorem \ref{parkthm} and Example \ref{mvexample}, there are at most $3^3-1 =26$ zero-dimensional components of interest.  We are done by Lemmas \ref{bestuniversalbasis} and \ref{lem:rankonelemmadegreed}.
\end{proof}

\begin{lemma}\label{Fptworeductionboundcubic}
Let $C/\QQ$ be a hyperelliptic curve of genus $g \geq 4$, with a rational Weierstrass point, geometrically simple Jacobian with $r \leq 1$, good reduction at 3, which satisfies condition $(\dagger)$.  

Let $P_1, P_2 \in C(\Qbar)$ be conjugate quadratic points, with $\overline{P_1}, \overline{P_2} \in C_{\mathbb{F}_3}(\mathbb{F}_{9})\setminus C_{\mathbb{F}_3}(\mathbb{F}_3)$, and $P_3 \in C(\QQ)$ a rational point. If $n(\Lambda_C, \overline{P_1}) =1$, there are at most $26$ ordered triples $(Q_1, Q_2,Q_3)$ of conjugate cubic points in $D_{\overline{P_1}} \times D_{\overline{P_2}} \times D_{\overline{P_3}}$.  If $n(\Lambda_C, \overline{P_1}) = 0$, there are no such triples.
\end{lemma}

\begin{proof}
In the first two coordinates, we can consider $\New_1$ using Remark \ref{valuationbound} and the fact that $\QQ_3(P_1,Q_1)$ is unramified, whereas we have to consider $\New_{1/3}$ in the last coordinate. We also have that $n(\Lambda_C, \overline{P_1}) = n(\Lambda_C, \overline{P_2})$.  Theorem \ref{stollbound} asserts that $n(\Lambda_C,\overline{P_1}) + n(\Lambda_C,\overline{P_2}) + n(\Lambda_C,\overline{P_3}) \leq 2$, so we have two cases. If $n(\Lambda_C,\overline{P_1}) =  0$, then there are no ordered triples by first statement of Lemma \ref{newtonbounds}. If $n(\Lambda_C,\overline{P_1}) = 1$, then the result follows from the second statement of Lemma \ref{newtonbounds}, Lemma \ref{newtonboundscubic}, and the same computation as in Lemma \ref{Fpreductionboundcubic}. 
\end{proof}

\begin{lemma}\label{Fpthreereductionboundcubic}
Let $C/\QQ$ be a hyperelliptic curve of genus $g \geq 4$, with a rational Weierstrass point, geometrically simple Jacobian with $r \leq 1$, good reduction at 3, which satisfies condition $(\dagger)$. 

 Let $P_1, P_2, P_3 \in C(\Qbar)$ be three conjugate cubic points, with $\overline{P_1}, \overline{P_2}, \overline{P_3} \in C_{\mathbb{F}_3}(\mathbb{F}_{27})\setminus C_{\mathbb{F}_3}(\mathbb{F}_3)$. Then there are no ordered triples $(Q_1, Q_2,Q_3)$ of conjugate cubic points in $D_{\overline{P_1}} \times D_{\overline{P_2}} \times D_{\overline{P_3}}$ not equal to $(P_1, P_2,P_3)$.
\end{lemma}

\begin{proof}
The proof is similar to that of Lemma \ref{Fpreductionboundcubic} with two changes.  First, one can consider $New_1$ instead of $New_{1/3}$ in all the coordinates, again using Remark \ref{valuationbound} and the fact that $Q_3(P_1,Q_1)$ is unramified. Second, note that since $P_1$, $P_2$, and $P_3$ reduce to (necessarily distinct) points outside of $C_{\mathbb{F}_3}(\mathbb{F}_3)$, we know that $P_1,\, P_2$, and $P_3$ remain conjugate over $\QQ_3$.  Thus their reductions are conjugate over $\mathbb{F}_3$, so we have $n(\Lambda_C, \overline{P_1}) = n(\Lambda_C, \overline{P_2}) = n(\Lambda_C, \overline{P_3})$.
%since $n(\Lambda_C, \overline{Q_i})$ is defined using forms over $\Qp$.  
By Theorem \ref{stollbound}, this common value can only be 0, and then the result follows from Lemma \ref{newtonbounds}. 
\end{proof}

\begin{proof}[Proof of Theorem \ref{cubicthm}]
For a given value of $g$, the family given by Proposition \ref{simple} and Lemma \ref{goodfamily} comprises a positive proportion of all hyperelliptic curves with a rational Weierstrass point, since the latter is defined by finitely many congruence conditions.  By Corollary \ref{goodrankcor}, at least 25\% of the curves in this family have rank $r \leq 1$; this is still a positive proportion of all the curves.  
Furthermore, under Assumption \ref{ass:Park}, a positive proportion of these curves satisfy condition $(\dagger)$. 
%When $g=5$, this proportion does not change if we restrict to non-trielliptic curves, by Lemma \ref{negligibleelliptic}.

Let $C$ be such a curve.  
Any triple of conjugate cubic points on $C$ that reduce to $\mathbb{F}_3$-points will have to lie in $D_{\overline{\infty}} \times D_{\overline{\infty}} \times D_{\overline{\infty}}$.  We may choose $P_1 = P_2 = P_3  = \infty$ and apply Lemma \ref{Fpreductionboundcubic} to conclude there are at most 26 ordered triples $(Q_1, Q_2,Q_3)$ of conjugate cubic points in this residue class.  But since $\overline{Q_1}=\overline{Q_2} = \overline{Q_3}$, each \textit{unordered} triple is overcounted by a factor of $6$, so there are at most $\floor{26/6} = 4$ unordered triples of such cubic points.

If the minimal Weierstrass model of $C$ is $y^2=f(x)$, note that (for some choice of square root) the pair of quadratic points $(i, \pm \sqrt{f(i)})$ reduces to $(\overline{i}, \pm \alpha)$ for each $\overline{i}=0,1,2.$  
Any triple of conjugate cubic points on $C$ where two of the points reduce to $(\mathbb{F}_9\setminus\mF_3)$-points will have to lie in $D_{(\overline{i},\alpha)} \times D_{(\overline{i},-\alpha)} \times D_{\overline{\infty}}$ for some $\overline{i}$. 
In $D_{(\overline{i},\alpha)} \times D_{(\overline{i},-\alpha)} \times D_{\overline{\infty}} $, we may choose $P_1 = (i, \sqrt{f(i)}), \, P_2 = (i, -\sqrt{f(i)}),$ and $P_3 = \infty$  and apply Lemma \ref{Fptworeductionboundcubic} with the values $n(\Lambda_C, (\overline{i},\alpha))$, $n(\Lambda_C, (\overline{i},-\alpha))$, and $n(\Lambda_C, \overline{\infty})$ to count ordered triples. 
%If $n(\Lambda_C, (\overline{i},\alpha)) = 0$, we conclude there are no other triples in this residue class, and if $n(\Lambda_C, (\overline{i},\alpha)) = 1$, there are at most 26 ordered triples. 

The last case is when a triple of conjugate cubic points on $C$ reduces to (necessarily distinct) $(\mathbb{F}_{27}\setminus\mF_3)$-points. In this setting, Lemma \ref{Fpthreereductionboundcubic} asserts that there are no triples in this residue class away from their centers. Since $C$ is hyperelliptic and has good reduction at 3, any \textit{unordered} triple of conjugate cubic points  (over $\mF_3$) will have to lie over an unordered triple of cubic points of $\mP^1_{\mathbb{F}_3}$, of which there are only $((3^3 + 1)-(3+1))/3 = 8$.

%here are, a priori, at most $2(3^3 - 1) - 1 = 55$ points on $C(\mF_{27})\setminus C(\mF_3)$, but that means there are actually at most $54$ as they come partitioned into conjugate triples. Every \textit{unordered} triple of conjugate cubic points can be arranged to lie over one of at most $54/3 = 18$ ordered triples of $(\mF_{27}\setminus\mF_3)$-points. 

%Unlike the proof of Theorem \ref{hyperthm}, we need to account for certain residue classes for $C$.
%First, we do not have any contribution from residue classes coming from $\mF_3$ or $\mF_9\setminus\mF_3$-points as we can use non-cubic points as the centers for these residue polydisks. 

Using Theorem \ref{stollbound}, we get the worst bound on the number of 0-dimensional components in all residue classes by assuming that for some $\overline{i}=0,1,2$, $n(\Lambda_C, (\overline{i},\alpha)) = n(\Lambda_C, (\overline{i},-\alpha))= 1$, and all the others will be 0. To conclude, there are at most $$4 + 26 + 8 = 38$$ unordered triples of conjugate cubic points. 
\end{proof}

\section*{Acknowledgments} 
We thank Johan de Jong, Bastian Haase, Joe Kramer-Miller, Aaron Landesman, Jennifer Park, Bjorn Poonen, Joe Rabinoff, and David Zureick-Brown for helpful discussions. 
We also thank Eric Katz and Jerry Wang for useful comments on an earlier draft. 
This research began at the ``Chip Firing and Tropical Curves" summer graduate school, organized by Matt Baker, David Jensen, and Sam Payne at the Mathematical Sciences Research Institute. The first author's work was partially done with the support of National Science Foundation grant DMS-1502553.  Also via the first author, this project has received funding from the European Research Council (ERC) under the European Union's Horizon 2020 research and innovation programme (grant agreement No 715747).

\bibliographystyle{amsalpha}
\def\bibfont{\small}
\bibliography{GMbib.bib} 
\end{document}